\documentclass[11pt]{article}
\usepackage{amsfonts}
\usepackage{graphicx}
\usepackage{epsfig}
\usepackage{amssymb, amsmath}
\usepackage{amsthm}
\usepackage{setspace}
\usepackage[top = 1in, bottom = 1in, left = 1.2in, right =
1.2in]{geometry}

\newtheorem{theorem}{Theorem}

\newtheorem{definition}[theorem]{Definition}

\newtheorem{proposition}[theorem]{Proposition}
\newtheorem{remark}[theorem]{Remark}

\numberwithin{equation}{section}
\newcommand \smu {\sqrt{\mu}}
\newcommand \del {\partial}

\begin{document}

\title{On the Vlasov-Poisson-Fokker-Planck equation \\
near Maxwellian}
\author{Hyung Ju Hwang\footnote{Supported by Basic Science Research Program
(2010-0008127) and Priority Research Centers Program (2010-0029638) through the National Research Foundation of Korea (NRF).}\, and
Juhi Jang\footnote{Supported by NSF Grant DMS-0908007}}

\maketitle


\begin{abstract}
We establish the exponential time decay rate of smooth solutions of small
amplitude to the Vlasov-Poisson-Fokker-Planck equations to the Maxwellian both
in the whole space and in the periodic box via the uniform-in-time energy
estimates and also the macroscopic equations. 
\end{abstract}

\bigskip

\numberwithin{equation}{section}

\section{Introduction}

In this article, we study the convergence to the equilibrium for the
nonlinear Vlasov-Poisson-Fokker-Planck  (VPFP)  equations in the whole space.
The VPFP system is one of the fundamental kinetic models in plasma physics to
describe the dynamics of charged particles (electrons and ions) subject to
the electrostatic force coming from their Coulomb interaction and to a
Brownian force modeling their collisions (Fokker-Planck). The VPFP system reads
\begin{equation}
f_{t}+v\!\cdot \!\nabla _{x}f+\text{div}_{v}((E-\beta v)f)=D\Delta _{v}f
\label{VPFP}
\end{equation}%
where $f=f(t,x,v)\geq 0$ is the distribution of particles at time $t$,
position $x$ and velocity $v$ for $(t,x,v)\in \mathbb{R}_{+}\times \mathbb{R}%
^{3}\times \mathbb{R}^{3}$, $\beta >0$ is the friction coefficient, $D>0$ is
the thermal diffusion coefficient, and $E=E(t,x)$ is the self-consistent
electric force. The equation \eqref{VPFP} is coupled with the Poisson
equation
\begin{equation}
E=-\nabla \Phi ,\quad -\Delta \Phi =\int_{\mathbb{R}^{3}}fdv-1,\quad \Phi
\longrightarrow 0\text{ as }|x|\rightarrow \infty   \label{poisson}
\end{equation}%
where $\Phi =\Phi (t,x)$ is the internal potential.

The global equilibrium of \eqref{VPFP} and \eqref{poisson} is given by the Maxwellian:
\begin{equation}
f=\mu :=\mu (v):=(\frac{2\pi D}{\beta })^{-\frac{3}{2}}e^{-\frac{\beta
|v|^{2}}{2D}},\quad E=0\,.
\end{equation}%
For simplicity, we take $\beta =1,\;D=1$. Letting
\begin{equation*}
f=\mu +\sqrt{\mu }g\,,
\end{equation*}%
the VPFP system is written in the following perturbed form: 
\begin{equation}
\begin{split}
& g_{t}+v\!\cdot \!\nabla _{x}g+v\sqrt{\mu }\!\cdot \!\nabla _{x}\phi +(%
\frac{v}{2}g-\nabla _{v}g)\!\cdot \!\nabla _{x}\phi -Lg=0, \\
& -\Delta \phi =\int_{\mathbb{R}^{3}}g\sqrt{\mu }dv ,
\end{split}%
\label{pVPFP}
\end{equation}%
where $L$ is the linearized Fokker-Planck operator given by 
\begin{equation}\label{L}
-Lg:=(\frac{|v|^{2}}{4}-\frac{3}{2}-\Delta _{v})g=-\frac{1}{\sqrt{\mu }}%
\nabla _{v}\!\cdot \!\left[ \mu \nabla _{v}(\frac{g}{\sqrt{\mu }})\right]\,.
\end{equation}%
The Fokker-Planck operator is a well-known hypoelliptic operator.
Diffusion in $v$ together with transport $v \cdot \nabla_x $ has a
regularizing effect not only in $v$ but also in $t$ and $x$. Note
that this is nontrivial since the diffusion only acts on the
velocity. This phenomenon can be obtained by applying
H\"{o}rmander's commutator (cf. \cite{H}) to linear Fokker-Planck
operator. For more details, we refer to \cite{B}. In particular, in \cite{HN}  smoothing
property has been shown for the linear Vlasov-Fokker-Planck  as a
hypoelliptic operator when the external potential satisfies a
certain condition.

On the other hand, the Vlasov-Fokker-Planck (VFP) operator is also known as hypocoercive operator,
which concerns rate of convergence to equilibrium. (See \cite{V}). Indeed, the trend to the equilibrium
with the rate of $t^{- 1/\epsilon}$
is investigated in \cite{DV} in the case of the linear VFP equation with the external potential
which is strictly convex at infinity.
Later, it is proved in \cite{HN} that
the solutions for the linear VFP equation with the external potential of high-degree
approach exponentially to the Maxwellian equilibrium and the rate was given explicitly.
In \cite{V}, Villani shows almost exponential decay to the equilibrium for the weakly self-consistent
VFP equation in the periodic box
with the Coulomb interaction potential replaced by a small and smooth potential.
When the coupling of the Poisson equation comes into play, then the fully nonlinear VPFP system adds more difficulties
to study in this context.
In \cite{BD}, the convergence of free energy solutions in $L^1(\mathbb{R}^6)$ to the equilibrium for the VPFP equations with
the confinement by the external potential was studied, but the convergence rate was not considered.

To our knowledge, the convergence to the Maxwellian has not been yet shown for the fully nonlinear
VPFP equations in the whole space without any confinement. The
main goal of this paper is to establish the exponential convergence rate of the solutions for the nonlinear VPFP equations 
to the  Maxwellian in the whole space  in the regime of small and smooth solutions.

Before we state the main results, we briefly review the existence
theories for the VPFP equations. Global existence of the solutions to the
VPFP system in two and three dimensions have been studied by many authors
and we will not attempt to exhaust references in this paper.
Classical solutions were obtained by Victory and O'Dwyer \cite{VO}
in two dimensions and by Rein and Weckler \cite{RW} for small data
in three dimensions.  Bouchut \cite{B2}  established the existence
and uniqueness of a global smooth solution in $L^1$ setting in
three dimensions. Asymptotic behaviors and time decay  of the
solutions near vacuum regime have been considered by Caprio
\cite{C}, Carrillo, Soler and Vazquez \cite{CSV}, and Ono and
Strauss \cite{OS}. We mention the work of Victory \cite{Vi},
Carrillo and Soler \cite{CS} where the global weak solutions were
constructed and also Bouchut \cite{B3} where the smoothing effect
was observed. Recently, stability of the front under a
Vlasov-Fokker-Planck dynamics was studied in \cite{EGM}.

There have been a lot of recent progress made on the results
towards the convergence rate to the equilibrium near Maxwellian regime
for other fundamental kinetic models such as the Boltzmann
equations with various collision kernels, for instance see
\cite{DS, DV2, M, SG}. We also mention that there are quite recent
results on this problem of trend to the Maxwellian equilibrium for the
fluid-kinetic models such as Navier-Stokes-Vlasov-Fokker-Planck
\cite{GHMZ} and Vlasov-Fokker-Planck-Euler equations \cite{CDM}.

In the next section, we introduce some notations and state the main results.

\section{Notations and Main Results}\label{2}

Here are some notations which will be used throughout the paper.
\begin{equation*}
\begin{split}
&\langle f,g\rangle:=\int_{\mathbb{R}^3} fgdv ,\\
&|g|^2_\nu:=\int_{\mathbb{R}^3}\left( |\nabla_vg|^2+\nu(v)|g|^2\right)dv
\,\text{ where }\, \nu(v):=1+|v|^2 ,\\
&\|g\|^2_\nu:=\int_{\mathbb{R}^3\times\mathbb{R}^3}\left(
|\nabla_vg|^2+\nu(v)|g|^2\right)dxdv ,\\
&\|g\|^2:=\|g\|^2_{L^2(\mathbb{R}_x^3\times\mathbb{R}_v^3)} \text{ or }
\|a\|^2:=\|a\|^2_{L^2(\mathbb{R}_x^3)},\\
& \|g\|_\infty:=\|g\|_{L^\infty(%
\mathbb{R}_x^3\times\mathbb{R}_v^3)} ,\\
&\del^\alpha_x:=\del^{\alpha_1}_{x_1}\del^{\alpha_2}_{x_2}\del%
^{\alpha_3}_{x_3} .
\end{split}%
\end{equation*}
We recall the basic properties of the linearized Fokker-Planck operator $L$ in \eqref{L},
for instance see \cite{CDM,V}.
\begin{equation}  \label{propL}
\langle f, Lg\rangle =\langle Lf, g\rangle,\quad\text{Ker}\,L=\text{Span}\{%
\sqrt{\mu}\},\quad L(v\sqrt{\mu})=-v\sqrt{\mu} .
\end{equation}
We further introduce projections onto $\smu$ and $v\smu$.
\begin{equation*}
\begin{split}
&\mathbf{P}_0 g:=\langle
g,\sqrt{\mu}\rangle\sqrt{\mu}=:\sigma(t,x)\sqrt{\mu} ,
\\
&\mathbf{P}_1 g:=\langle g,v\sqrt{\mu}\rangle\!\cdot\! v\sqrt{\mu}=:{u}%
(t,x)\!\cdot\! v\sqrt{\mu} ,\\
&\mathbf{P}:=\mathbf{P}_0+\mathbf{P}_1 .
\end{split}%
\end{equation*}
Here $\sigma$ represents the density fluctuation and $u$ the velocity
fluctuation. Now the coercivity of $-L$ can be written in terms of
projections as follows (for instance see \cite{CDM}): for a positive
constant $\lambda_0>0$
\begin{equation}
\begin{split}  \label{coerL}
\langle g, -Lg \rangle &\geq \lambda_0|\{\mathbf{I-P}_0\}g
|^2_\nu\quad ,
\text{ or} \\
\langle g, -Lg \rangle &\geq \lambda_0|\{\mathbf{I-P}\}g
|^2_\nu+|{u}|^2 .
\end{split}%
\end{equation}

We next define the instant energy functionals and dissipation.

\begin{definition}[Instant energy]
For $N\geq 3$, there exists a constant $C>0$ such that an instant energy
functional $\mathcal{E}_{N}(g,\nabla \phi )=:\mathcal{E}_{N}(t)$ satisfies the following:
\begin{equation*}
\frac{1}{C}\mathcal{E}_{N}(t)\leq \sum_{|\alpha |\leq N}\left( \Vert \del%
_{x}^{\alpha }g\Vert ^{2}+\Vert \del_{x}^{\alpha }\nabla \phi
\Vert ^{2}\right) \leq C\mathcal{E}_{N}(t) 
\end{equation*}%
where $\mathcal{E}_{N}(t)$ will be defined in (\ref{EN}).
\end{definition}

\begin{definition}[Dissipation]
For $N\geq 3$, the dissipation rate $\mathcal{D}_N(t)$ is defined by
\begin{equation*}
\begin{split}
\mathcal{D}_N(t)&:= \frac12\sum_{|\alpha|\leq N} \left( \|
\mathbf{P}_1\del_x^\alpha g\|^2+\lambda_0\|\{\mathbf{I-P} \}\del_x^\alpha g\|_\nu^2\right) \\
&\quad+ \frac{\kappa}{2} \sum_{|\alpha|\leq N-1}\left(\|
\del_x^\alpha\nabla \phi \|^2+ 4\|\mathbf{P}_0\del_x^\alpha g\|^2
+\|\mathbf{P}_0\del_x^\alpha \nabla g\|^2  \right) ,
\end{split}%
\end{equation*}
where $\kappa=\min\{\lambda_0/2,1/8\}$.
\end{definition}

\begin{remark}
Note that the first part of the dissipation $\mathcal{D}_{N}$ is due to the
collisions as seen in \eqref{coerL}. What is distinguishable for the VPFP system is that
the macroscopic part as well as the potential part also dissipates because of the extra damping due to
Coulomb interaction. The exponential convergence rate is obtained through
such damping of the density fluctuation, which can be acheived by means of the
macroscopic equations.
\end{remark}

We are now ready to state the main results. The first theorem is about the
uniform energy estimates for the VPFP equations.

\begin{theorem}
\label{thm1} Let $N\geq 3$. There exist an instant energy functional $%
\mathcal{E}_N(t)$ and a sufficiently small $\epsilon_0>0$ so that if $%
\mathcal{E}_N(0)\leq \epsilon_0$, then the smooth solutions $(g,\nabla\phi)$
to the Vlasov-Poisson-Fokker-Planck equations \eqref{pVPFP} satisfy the
following energy inequality
\begin{equation}  \label{eeee}
\frac{d}{dt}{\mathcal{E}_N}(t) + \mathcal{D}_N(t) \leq 0\,.
\end{equation}
In particular, we have the global energy bound
\begin{equation*}
\sup_{t\geq 0}{\mathcal{E}_N}(t) \le {\mathcal{E}_N}(0)\,.
\end{equation*}
\end{theorem}

The next theorem is about the exponential convergence to the equilibrium of 
the VPFP system. 

\begin{theorem}
\label{thm2} There exist $\epsilon _{0}>0$ and $\eta >0$ such that for small
initial data $\mathcal{E}_{N}(0)\leq \epsilon _{0}$, the solutions decay
exponentially
\begin{equation}\label{decay}
\mathcal{E}_{N}(t)\leq {\mathcal{E}_N}(0)e^{-\eta t} ,
\end{equation}
where $\eta$ can be chosen as $ \frac{2\kappa}{5}$.
\end{theorem}

Theorem \ref{thm2}, where an explicit convergence rate is given, is a direct
consequence of Theorem \ref{thm1} and here we provide the proof of it.  From the definition of $\mathcal{E}_N$ -- see (\ref{EN}), we first note that
 \[
 \mathcal{E}_N(t) \leq \frac{5}{4} \sum_{|\alpha |\leq N}\left( \Vert \del
_{x}^{\alpha }g\Vert ^{2}+\Vert \del_{x}^{\alpha }\nabla \phi \Vert
^{2}\right).
\]
Next,  from the Poisson equation in \eqref{pVPFP}, we deduce that when $|\alpha|=N$, $\|\del_x^\alpha \nabla\phi\|^2\leq \| \mathbf{P}_0\del_x^{\beta} g\|^2$ for $|\beta|=N-1$. Hence by the above definition of $\mathcal{D}_N$, we see that
\[
\frac{2\kappa}{5} \mathcal{E}_N(t) \leq \frac{\kappa}{2} \sum_{|\alpha |\leq N}\left( \Vert \del
_{x}^{\alpha }g\Vert ^{2}+\Vert \del_{x}^{\alpha }\nabla \phi \Vert
^{2}\right) \leq \mathcal{D}_N.
\]
 Therefore, the exponential decay \eqref{decay} follows from the energy inequality \eqref{eeee}  by letting
$\eta=\frac{2\kappa}{5}$.

\begin{remark}
The above theorems are also valid for the periodic domains without any
changes in the proof. Our results show that the convergence rate to the
equilibrium of VPFP system  is the same exponential for both the periodic box and
the whole space in the framework of smooth solutions. This is not the case
for other kinetic models such as Boltzmann equations \cite{DS} and for fluid-kinetic models \cite{CDM, GHMZ}.
\end{remark}

\begin{remark}
The global existence of solutions to \eqref{VPFP} follows from the a priori
global energy bound by rather standard method. In this paper, we focus on
proving the uniform estimates.
\end{remark}

The proof of Theorem \ref{thm1} is based on the novel
uniform-in-time energy methods developed for the study of
Boltzmann equations over the years; for its original idea of the proof relevant to our model, for instance see \cite{G1} and also see \cite{LY}. It consists
of the instant energy estimates for $\mathcal{E}_N$ and the
 macroscopic equations. It is interesting to see how this mechanism through the macroscopic equations takes full advantage of the
self-consistent Poisson interaction, which is a key of getting the
exponential convergence. We remark that the method will not give
the exponential convergence rate for the linear
Vlasov-Fokker-Planck equation in the whole space without the
Poisson interaction.

\section{Energy estimates}\label{3}

In this section, we will derive the energy estimates for the VPFP system. The first part is on the instant energy estimates and the second part is on recovering the full dissipation of the perturbation via the macroscopic equations. 

\subsection{Instant energy estimates}

The goal of this subsection is to prove the following instant energy inequality:

\begin{proposition}
\label{prop1} There exists a constant $C>0$ independent of $t$ such that the
smooth solutions $(g,\nabla \phi )$ to the Vlasov-Poisson-Fokker-Planck
equations \eqref{pVPFP} satisfy the following energy inequality
\begin{equation}
\frac{1}{2}\frac{d}{dt}\widetilde{\mathcal{E}}_{N}(t)+\widetilde{\mathcal{D}}%
_{N}(t)\leq C({\widetilde{\mathcal{E}}_{N}(t)})^{\frac{1}{2}}\mathcal{D}%
_{N}(t) , \label{ee}
\end{equation}%
where
\begin{equation}
\begin{split}
\widetilde{\mathcal{E}}_{N}(t)& :=\sum_{|\alpha |\leq N}\left( \Vert \del%
_{x}^{\alpha }g\Vert ^{2}+\Vert \del_{x}^{\alpha }\nabla \phi \Vert
^{2}\right) \text{ and} \\
\widetilde{\mathcal{D}}_{N}(t)& :=\sum_{|\alpha |\leq N}\left( \Vert \del%
_{x}^{\alpha }u\Vert ^{2}+\lambda _{0}\Vert \{\mathbf{I-P}\}\del_{x}^{\alpha
}g\Vert _{\nu }^{2}\right) .
\end{split}
\label{ENthilda}
\end{equation}
\end{proposition}

\begin{proof} 
First, we project \eqref{pVPFP} onto $\{\sqrt{\mu}\}$ to get the conservation of the density fluctuation 
\begin{equation}\label{a}
\sigma_t+\nabla\!\cdot\! {u}=0,\quad -\Delta\phi= \sigma 
\end{equation}
which will be frequently used throughout the argument. We will start with the case of $\alpha=0$ in \eqref{ee}: multiply \eqref{pVPFP} by $g$ and integrate over
$\Omega\times\mathbb{R}^3$ to get
\[
\frac12\frac{d}{dt}\int g^2 dxdv+\underbrace{\int
gv\sqrt{\mu}\!\cdot\!\nabla_x\phi dxdv}_{(i)} +\underbrace{\int
g(\frac{v}{2}g-\nabla_v g)\!\cdot\!\nabla_x\phi dxdv}_{(ii)}{-\int
gLgdxdv}=0 .
\]
For $(i)=\int {u}\cdot \nabla\phi dx $, we use the integration by parts and  \eqref{a} to get
\[
(i)=-\int (\nabla\cdot u)\phi dx = \int \sigma_t\phi dx =-\int \Delta\phi_t \phi dx = \int \nabla\phi_t\cdot\nabla\phi dx
\]
and thus we derive that 
\[
(i)=\frac12\frac{d}{dt}\int  |\nabla\phi|^2dx .
\]
For $(ii)$, we use the macroscopic variables and the microscopic part to rewrite $g$: 
\[
\begin{split}
(ii)=\int \frac{g^2}{2} v\cdot\nabla\phi dxdv= \int \frac{1}{2}
(\sigma\smu+u\cdot v\smu+\{\mathbf{I-P}\}g)^2 v\cdot\nabla\phi dx
dv ,
\end{split}
\]
then by taking the $L^\infty$ of the potential, 
we get
\[
|(ii)|\precsim \|\nabla\phi\|_\infty(\|\sigma\|^2_{L^2}+\|u\|^2_{L^2}+\|\{\mathbf{I-P}\}g\|^2_\nu)\,.
\]
With \eqref{coerL}, we obtain the following: 
\begin{equation}
\begin{split}
\frac12\frac{d}{dt}\left(\|g\|^2+\|\nabla\phi\|^2\right) +&\lambda_0\|\{\mathbf{I-P}\}g\|^2_\nu +\|u\|^2\\
\precsim  \|\nabla\phi\|_\infty&(\|\sigma\|^2+\|u\|^2+\|\{\mathbf{I-P}\}g\|^2_\nu)\,.
\end{split}
\end{equation}

We next derive  the higher-order estimates. Let  $|\alpha|\leq N$.
Take $\del^\alpha$ of \eqref{pVPFP}
\[
\del_x^\alpha g_{t}+v\cdot\nabla_x\del_x^\alpha g+v\sqrt{\mu}\!\cdot\!\del_x^\alpha\nabla_x\phi+\del^\alpha_x\left[(\frac{v}{2}g-\nabla_v g)\!\cdot\!\nabla_x\phi\right]-L\del_x^\alpha g=0. 
\]
The linear terms can be treated in the same way as in the case of $\alpha=0$. Hence, by multiplying by $\del_x^\alpha g$ and integrating we get
\[
\begin{split}
&\frac12\frac{d}{dt}\left(\|\del_x^\alpha g\|^2+\|\del_x^\alpha\nabla\phi\|^2\right) +\lambda_0\|\{\mathbf{I-P}\}\del_x^\alpha g\|^2_\nu +\|\del_x^\alpha u\|^2\\
&\leq -\sum_{|\beta|\leq |\alpha|}C_\beta\int \del^\alpha_xg \left[(\frac{v}{2}\del_x^{\alpha-\beta}g-\nabla_v \del_x^{\alpha-\beta}g)\!\cdot\!\del^\beta_x\nabla_x\phi\right] dxdv\\
&= \sum_{|\beta|\leq [\frac{|\alpha|}{2}]} +
\sum_{|\beta|>[\frac{|\alpha|}{2}]} =: (I)+(II) ,
\end{split}
\]
where $[\gamma]$ denotes the Gauss number i.e. the greatest integer less than or equal to $\gamma$.

Now by using the identity $\del_x^{\alpha-\beta}g=\del_x^{\alpha-\beta} \sigma\smu+\del_x^{\alpha-\beta}u\cdot v\smu+\{\mathbf{I-P}\}\del_x^{\alpha-\beta}g$ and by Cauchy-Schwartz inequality, we see that
\[
\begin{split}
|(I)|\precsim  \sum_{|\beta|\leq [\frac{|\alpha|}{2}]}\| \del_x^\beta\nabla_x\phi \|_\infty \big(&\| \del^\alpha_x\sigma\|^2
+\|\del_x^\alpha u\|^2+\|\{\mathbf{I-P}\}\del^\alpha_x g\|^2_\nu \\
&+\| \del^{\alpha-\beta}_x\sigma\|^2
+\|\del_x^{\alpha-\beta} u\|^2+\|\{\mathbf{I-P}\}\del^{\alpha-\beta}_x g\|^2_\nu \big)
\end{split}
\]
and
\[
\begin{split}
|(II)|\precsim  \sum_{|\beta|>[\frac{|\alpha|}{2}]} \big(   \| \del^{\alpha-\beta}_x\sigma\|_\infty
&+\|\del_x^{\alpha-\beta} u\|_\infty+\sup_{x} \big|\{\mathbf{I-P}\}\del^{\alpha-\beta}_x g \big|_{L^2_v} \big)  \\
\cdot\big( \| \del^\alpha_x\sigma&\|^2
+\|\del_x^\alpha u\|^2+\|\{\mathbf{I-P}\}\del^\alpha_x g\|^2_\nu +\|\del_x^\beta \nabla_x\phi  \|^2 \big)\,.
\end{split}
\]
Since $N\geq 3$, Sobolev embedding yields 
\[
 \sum_{|\beta|\leq [\frac{|\alpha|}{2}]} \| \del_x^\beta\nabla_x\phi \|_\infty +   \sum_{|\beta|>[\frac{|\alpha|}{2}]}  \big(\| \del^{\alpha-\beta}_x\sigma\|_\infty
 +\|\del_x^{\alpha-\beta} u\|_\infty+\sup_{x} \big|\{\mathbf{I-P}\}\del^{\alpha-\beta}_x g \big|_{L^2_v} \big)\precsim ({\widetilde{\mathcal{E}}_N})^{\frac12}.
\]
Therefore, we obtain
\[
\frac12\frac{d}{dt}{\widetilde{\mathcal{E}}_N}
+\widetilde{\mathcal{D}}_N\precsim({\widetilde{\mathcal{E}}_N})^{\frac12}
\mathcal{D}_N ,
\]
where we have used the elliptic regularity of the Poisson equation in \eqref{pVPFP}
\begin{equation}\label{er}
\|\phi\|_{H^{s+2}} \leq C \| \sigma \|_{H^s} \text{ for }s\geq 0\,.
\end{equation}
This completes the proof of Proposition \ref{prop1}.
\end{proof}

\subsection{Macroscopic equations}

In this subsection, we will prove that the macroscopic part $\mathbf{P}_0g$ as well as the potential part $\nabla\phi$ of the solution to the VPFP system 
also dissipates due to the extra damping through the Coulomb interaction.
This will be established via so-called the method of macroscopic equations. 

\begin{proposition}
\label{prop2} There exists a constant $C>0$ independent of $t$ such that the
smooth solutions $(g,\nabla \phi )$ to the Vlasov-Poisson-Fokker-Planck
equations \eqref{pVPFP} satisfy the following energy inequality
\begin{equation}
\begin{split}
 \frac{d}{dt}G(t)&+\sum_{|\beta |\leq N-1}\left( \frac{1}{2}\Vert  \del_{x}^{\beta }\nabla\phi \Vert
^{2}+2\Vert \del_{x}^{\beta
}\sigma \Vert ^{2}+\frac{1}{2}\Vert \del_{x}^{\beta } \nabla\sigma \Vert
^{2}\right)  \\
& \leq \sum_{|\beta |\leq N-1}\left(  \Vert \del_{x}^{\beta
}u\Vert ^{2}+\Vert \nabla \cdot \del_{x}^{\beta
}u\Vert ^{2}+\Vert \{\mathbf{I-P}\}\del_{x}\del_{x}^{\beta
}g\Vert _{\nu }^{2}\right) \\
&\quad+C(\widetilde{\mathcal{E}}_{N}(t))^{\frac{1}{2}%
}\sum_{|\beta |\leq N-1}\big(\Vert \del_{x}^{\beta }\nabla\phi
\Vert ^{2}+\Vert \del_{x}^{\beta }\sigma \Vert ^{2} \big) ,
\end{split}%
\label{eee}
\end{equation}%
where
\begin{equation}
G(t)=\sum_{|\beta |\leq N-1}\int \big[\del_{x}^{\beta }u^{i}(\del_{x_{i}}\del%
_{x}^{\beta }\sigma +\del_x^\beta\del_{x_i}\phi)+\frac{1}{2}(|\del_{x}^{\beta }\nabla\phi |^{2}
+|\del_{x}^{\beta }\sigma |^{2} )\big]dx\,.
\label{G}
\end{equation}
\end{proposition}

\begin{proof} Our starting point of the proof is the macroscopic equation for $u=\langle g,v\sqrt{\mu}\rangle$, which will be derived shortly,  in addition to the macroscopic equation for $\sigma=\langle g,\sqrt{\mu}\rangle$:  \eqref{a}. 
First we note that since $\langle v^i\smu,v^j\smu\rangle=\delta^{ij}$ and $\langle v^iv^j\smu,v^k\smu\rangle=0$, 
\[
\begin{split}
\langle v^i\smu, v^j\del_{x_j}g \rangle&= \langle v^i\smu, v^j \big[\del_{x_j} \sigma\smu+ \del_{x_j}u^k\cdot v^k\smu+
\del_{x_j}\{\mathbf{I-P}\} g \big] \rangle \\
&=\del_{x_i}\sigma+ \del_{x_j} \langle v^iv^j \smu, \{\mathbf{I-P}\} g \rangle\\
\end{split}
\]
and 
\[
\begin{split}
\langle v^i\smu, (\frac{v^j}{2}g-\del_{v^j} g)\del_{x_j}\phi\rangle& =\frac12\langle v^iv^j\smu, \big[ \sigma\smu+u^k\cdot v^k\smu+\{\mathbf{I-P}\}g\big]\rangle \del_{x_j}\phi \\
&\quad+ \langle \delta^{ij}\smu-\frac12 v^iv^j\smu , \big[ \sigma\smu+u^k\cdot v^k\smu+\{\mathbf{I-P}\}g\big]\rangle \del_{x_j}\phi  \\
&=  \langle \delta^{ij}\smu, \big[ \sigma\smu+u^k\cdot
v^k\smu+\{\mathbf{I-P}\}g\big]\rangle \del_{x_j}\phi \\
&= \sigma
\del_{x_i}\phi .
\end{split}
\]
Hence by further using \eqref{propL}, the projection of the VPFP equation \eqref{pVPFP} onto $\{v^i\sqrt{\mu}\}$ can be recorded as follows:
\begin{equation}\label{b}
u^i_t+\del_{x_i}\sigma +\del_{x_i}\phi +\sigma\del_{x_i}\phi +u^i+
\del_{x_j}\langle v^iv^j\smu,\{\mathbf{I-P}\}g\rangle=0
\end{equation}
which is the macroscopic equation for $u$ which is not decoupled from the microscopic part. The idea is to estimate $\sigma$ and $\nabla\phi$ via their elliptic coupling and the dissipation of the microscopic part $\{\mathbf{I-P}_0\}g$. We will prove \eqref{eee} first when $\beta=0$. To obtain the estimate of $\sigma=\mathbf{P}_0g$ part, we multiply
\eqref{b} by $\del_{x_i}\sigma$ and integrate: 
\begin{equation}
\begin{split}
\int u_t^i \del_{x_i}\sigma dx +\int |\del_{x_i}\sigma|^2dx +\int \del_{x_i}\phi \del_{x_i}\sigma dx +\int \sigma\del_{x_i}\phi \del_{x_i}\sigma dx +\int u^i\del_{x_i}\sigma dx\\
+\int  \del_{x_j}\langle v^iv^j\smu,\{\mathbf{I-P}\}g\rangle
\del_{x_i}\sigma dx=0. 
\end{split}\label{3.9}
\end{equation}
We denote it by $(I)+(II)+(III)+(IV)+(V)+(VI)=0$. We do the integration by parts and use \eqref{a} to estimate $(I)$, $(III)$, $(IV)$, $(V)$ as follows: 
\[
\begin{split}
(I)&=\frac{d}{dt}\int u^i\del_{x_i}\sigma dx -\int u^i \del_{x_i}\sigma_t dx=\frac{d}{dt}\int u^i\del_{x_i}\sigma dx +\int \del_{x_i} u^i \sigma_t dx\\
\
&=\frac{d}{dt}\int u^i\del_{x_i}\sigma dx -\int |\nabla\cdot u|^2dx, \\
\
(III)&=-\int \Delta \phi \,\sigma \,dx =\int \sigma^2\,dx,\quad \\
\
(IV)&=-\frac12\int \Delta \phi \,\sigma^2 \,dx=\frac12\int \sigma^3\,dx, \\
\
(V)&=-\int \del_{x_i}u^i \sigma dx =\int \sigma_t \sigma dx=\frac12\frac{d}{dt}\int \sigma^2 dx.
\end{split}
\]
We also see that by Cauchy-Swartz inequality 
\[|(VI)|\leq \frac12\| \nabla \sigma \|^2 +\frac12 \| \{\mathbf{I-P}\} \del_x g\|^2_\nu\,.
\]
Hence we obtain the following: 
\begin{equation}
\begin{split}
\frac{d}{dt}\int \big( u^i\del_{x_i}\sigma &+\frac12 \sigma^2\big)dx  +\int \sigma^2dx+\frac12\int |\nabla \sigma|^2dx\\
&  \leq \frac12\int |\sigma|^3\,dx+\int |\nabla\cdot u|^2dx+ \frac12 \| \{\mathbf{I-P}\} \del_x g\|^2_\nu.
\end{split}\label{3.10}
\end{equation}

In order to get the estimate of $\nabla\phi$ part, we now multiply  \eqref{b} by $\del_{x_i}\phi$ and integrate.  Then we get the similar expression as \eqref{3.9}
\begin{equation*}
\begin{split}
\int u_t^i \del_{x_i}\phi dx +\int \del_{x_i}\sigma \del_{x_i}\phi dx +\int |\del_{x_i}\phi |^2 dx +\int \sigma|\del_{x_i}\phi |^2 dx +\int u^i\del_{x_i}\phi dx\\
+\int  \del_{x_j}\langle v^iv^j\smu,\{\mathbf{I-P}\}g\rangle
\del_{x_i}\phi dx=0 ,
\end{split}
\end{equation*}
denoted by $(i)+(ii)+(iii)+(iv)+(v)+(vi)=0$. The second term
$(ii)$ is the same as $(III)$ which gives rise to $\int
\sigma^2dx$. The third term $(iii)$ is a good damping term, the
fourth term $(iv)$ is a nonlinear term, and we will use
Cauchy-Schwartz inequality for $(vi)$ as before. The fifth term
$(v)$ forms the time integral as before:
\[
(v)=-\int \del_{x_i}u^i \phi dx =\int \sigma_t \phi
dx=\frac12\frac{d}{dt}\int |\nabla\phi|^2 dx .
\]
For the first term $(i)$, we do the integration by parts and use \eqref{a} to derive
\[
\begin{split}
(i)&=\frac{d}{dt}\int u^i\del_{x_i}\phi dx -\int u^i \del_{x_i}\phi_t dx=\frac{d}{dt}\int u^i\del_{x_i}\phi dx
+\int \del_{x_i} u^i \phi_t dx\\
\
&=\frac{d}{dt}\int u^i\del_{x_i}\phi dx -\int |\nabla\phi_t|^2dx.
\end{split}
\]
Instead of estimating the time derivatives directly, we rewrite
the second term by using the equation again: since
$\phi_t=-\Delta^{-1}\sigma_t=\Delta^{-1}(\nabla\cdot u)$,
\[
\int |\nabla\phi_t|^2dx = \int |\nabla \Delta^{-1}(\nabla\cdot
u)|^2dx \leq \int  |u|^2 dx ,
\]
where we have used the $L^2$ boundedness of $\nabla\Delta^{-1}\nabla\cdot\,$ which is a well-known singular integral operator:
Riesz potential \cite{St}.
Thus we get
\[
\begin{split}
\frac{d}{dt}&\int  \big(u^i\del_{x_i}\phi +\frac12|\nabla\phi|^2 \big)dx  +\int \sigma^2dx+\frac12\int |\nabla \phi|^2dx\\
&  \leq \int |\sigma| |\nabla\phi|^2\,dx+\int |u|^2dx+ \frac12 \| \{\mathbf{I-P}\} \del_x g\|^2_\nu.
\end{split}
\]
By combining it with \eqref{3.10}, we obtain
\[
\begin{split}
\frac{d}{dt}&\int  \big[u^i(\del_{x_i}\phi+\del_{x_i}\sigma) +\frac12(|\nabla\phi|^2+\sigma^2) \big]dx  +2\int \sigma^2dx+\frac12\int (|\nabla \phi|^2+|\nabla\sigma|^2)dx\\
&  \leq \int |\sigma| ( |\nabla\phi|^2+\sigma^2)dx+\int |u|^2dx+\int |\nabla\cdot u|^2dx+ \| \{\mathbf{I-P}\} \del_x g\|^2_\nu.
\end{split}
\]

We next handle the higher order estimates which are needed to
control nonlinear terms within our energy functionals. We will first  derive the estimates for
$\del_{x_i}\del_x^\alpha \sigma$  for  $|\alpha|\leq N-1$. Take
$\del_x^\alpha$ of \eqref{b}, multiply by $\del_{x_i}\del_x^\alpha
\sigma$  and integrate to get
\[
\begin{split}
\int  \del_x^\alpha u_t^i \del_{x_i}\del_{x}^\alpha \sigma dx +\int |\del_{x_i}\del_x^\alpha \sigma|^2dx +\int \del_{x_i}\del_x^\alpha\phi \del_{x_i}\del_x^\alpha \sigma dx +\underline{\int \del_x^\alpha[\sigma\del_{x_i}\phi] \del_{x_i}\del_x^\alpha \sigma dx}_{(\ast)} \\
+\int \del_x^\alpha u^i\del_{x_i}\del_x^\alpha \sigma dx
+\int  \del_{x_j}\del_x^\alpha \langle v^iv^j\smu,\{\mathbf{I-P}\}g\rangle \del_{x_i}\del_x^\alpha \sigma dx=0.
\end{split}
\]
We follow the same procedure as in $\alpha=0$ case. The linear terms can be treated in the same way. For the nonlinear terms $(\ast)$, note that
\[
\begin{split}
(\ast)&= \int \del_x^{\alpha} \sigma\del_{x_i}\phi \del_{x_i}\del_x^\alpha \sigma dx +\int \sigma\del_{x_i}\del_x^{\alpha}\phi \del_{x_i}\del_x^\alpha \sigma dx+\sum_{0<|\beta|<|\alpha|} C_\beta\int  \del_x^{\alpha-\beta} \sigma\del_{x_i}\del_x^{\beta}\phi \del_{x_i}\del_x^\alpha \sigma dx\\
&=\frac32\int \sigma|\del_x^\alpha \sigma|^2 dx-\int  \del_{x_i}\sigma\del_{x_i}\del_x^{\alpha}\phi\del_x^\alpha \sigma dx\\
&\quad-
\sum_{0<|\beta|<|\alpha|}
C_\beta\big[\int   \del_{x_i}\del_x^{\alpha-\beta} \sigma\del_{x_i}\del_x^{\beta}\phi\del_x^\alpha \sigma dx
- \int  \del_x^{\alpha-\beta} \sigma \del_x^{\beta}\sigma\del_x^\alpha \sigma dx\big]
\end{split}
\]
and thus by Sobolev embedding, we get
\[
|(\ast)|\precsim (\widetilde{\mathcal{E}}_N)^{\frac12} \sum_{|\beta|\leq |\alpha|} \|\del_x^\beta \sigma \|^2
\]
and in turn,
\[
\begin{split}
\frac{d}{dt}\int \big( \del_x^\alpha u^i\del_{x_i}\del_x^\alpha \sigma &+\frac12 |\del_x^\alpha \sigma|^2\big)dx  +\int |\del_x^\alpha \sigma|^2dx+\frac12\int |\nabla \del_x^\alpha \sigma|^2dx\\
&  \leq \int |\nabla\cdot \del_x^\alpha u|^2dx
+C(\widetilde{\mathcal{E}}_N)^{\frac12} \sum_{|\beta|\leq
|\alpha|} \|\del_x^\beta \sigma \|^2 +\frac12 \| \{\mathbf{I-P}\}
\del_x \del_x^\alpha g\|^2_\nu .
\end{split}
\]
 The estimates of $\del_{x_i}\del_x^\alpha \phi$ can be obtained in a similar way:
  \[
\begin{split}
\frac{d}{dt}\int \big( &\del_x^\alpha u^i\del_{x_i}\del_x^\alpha \phi +\frac12 |\del_x^\alpha \nabla\phi|^2\big)dx  +\int |\del_x^\alpha \sigma|^2dx+\frac12\int |\nabla \del_x^\alpha \phi|^2dx\\
&  \leq \int | \del_x^\alpha u|^2dx
+C(\widetilde{\mathcal{E}}_N)^{\frac12} \sum_{|\beta|\leq
|\alpha|} (\|\del_x^\beta \nabla\phi \|^2+ \|\del_x^\beta \sigma
\|^2) +\frac12 \| \{\mathbf{I-P}\} \del_x \del_x^\alpha g\|^2_\nu
.
\end{split}
\]
By adding the above two inequalities, we finish the proof of Proposition \ref{prop2}.
\end{proof}

\section{Proof of Theorem \protect\ref{thm1}}\label{4}

We are now ready to prove Theorem \ref{thm1}, based on Proposition \ref{prop1}
and \ref{prop2}.

\begin{proof}[Proof of Theorem \ref{thm1}] Let $\kappa=\min\{\lambda_0/2,1/8\}$. By combining \eqref{ee} and \eqref{eee}, we first deduce that there exists a constant $C>0$ so that
\[
\begin{split}
&\frac{d}{dt}\left(\widetilde{\mathcal{E}}_N(t)+ 2\kappa G(t)\right)  + \widetilde{\mathcal{D}}_N(t) +  2\kappa  \sum_{|\beta|\leq N-1}\left(2 \|\del_x^\beta \sigma\|^2+\frac12 \|\del_x^\beta\nabla \phi\|^2+\frac12 \|\nabla\del_x^\beta \sigma\|^2\right)
\\
&\quad \leq C ({\widetilde{\mathcal{E}}_N(t)})^{\frac12}
\mathcal{D}_N(t) ,
\end{split}
\]
where $\widetilde{\mathcal{E}}_N(t)$ is given in \eqref{ENthilda} and $G(t)$ in \eqref{G}.
Since $|G(t)|\leq  \widetilde{\mathcal{E}}_N(t)$, we see that $ \tfrac34  \widetilde{\mathcal{E}}_N(t)\leq \widetilde{\mathcal{E}}_N(t)+2 \kappa G(t) \leq \tfrac43 \widetilde{\mathcal{E}}_N(t) $.
Now we redefine an instant energy by
\begin{equation}
\mathcal{E}_N(t):=\widetilde{\mathcal{E}}_N(t)+ 2\kappa G(t).
\label{EN}
\end{equation}
Then by applying a standard continuity argument, we finally deduce  \eqref{eeee} by setting $\mathcal{E}_N$ sufficiently small initially.
\end{proof}\



\vskip 1cm
\noindent
{\sc Hyung Ju Hwang}\\
Department of Mathematics\\
Pohang University of Science and Technology \\
Pohang 790-784, Republic of Korea \\
{\tt hjhwang@postech.ac.kr}

\vskip 1cm
\noindent
{\sc Juhi Jang}\\
Department of Mathematics\\
University of California, Riverside\\
Riverside, CA 92521, USA \\
{\tt juhijang@math.ucr.edu}

\end{document}